\documentclass[11pt]{amsart}
\usepackage[T1]{fontenc}
\usepackage{lmodern}
\usepackage[margin=1in]{geometry}
\usepackage{mathtools,amssymb}
\usepackage{microtype}
\usepackage{xcolor}
\definecolor{linkblue}{RGB}{25,55,110}
\usepackage[colorlinks=true,linkcolor=linkblue,citecolor=linkblue,
            urlcolor=linkblue,pdfencoding=auto,psdextra]{hyperref}
\hypersetup{pdftitle={Oriented trees in Eulerian digraphs},
  pdfauthor={Dhruv Mubayi and Jacques Verstraete},
  pdfsubject={Sharp density conditions for oriented trees in Eulerian digraphs},
  pdfkeywords={Eulerian digraph, oriented tree, directed path, Erdos-Sos}}

\newtheorem{theorem}{Theorem}[section]
\newtheorem{lemma}[theorem]{Lemma}
\newtheorem{proposition}[theorem]{Proposition}

\newtheorem{conjecture}[theorem]{Conjecture}
\theoremstyle{definition}
\newtheorem{definition}[theorem]{Definition}
\theoremstyle{remark}

\numberwithin{equation}{section}
\newcommand{\E}{\mathbb E}
\newcommand{\Prob}{\mathbb P}
\newcommand{\ind}{\mathbf 1}
\newcommand{\supp}{\operatorname{supp}}
\newcommand{\emb}{\operatorname{emb}}
\newcommand{\cF}{\mathcal F}
\newcommand{\cA}{\mathcal A}
\newcommand{\cB}{\mathcal B}
\newcommand{\cC}{\mathcal C}
\newcommand{\doi}[1]{\href{https://doi.org/#1}{doi:\nolinkurl{#1}}}

\title[Erd\H os-S\'os for digraphs]{Erd\H os-S\'os for digraphs}
\author[Dhruv Mubayi,Jacques Verstra\"{e}te]{Dhruv Mubayi$^1$ and Jacques Verstra\"{e}te$^2$}
\address{$^1$ Department of Mathematics, Statistics and Computer Science, University of Illinois, Chicago, IL 60607. Email: mubayi@uic.edu. Research partially supported by NSF Awards DMS-2552740 and DMS-2153576.} 
\address{$^2$ Department of Mathematics, University of California, San Diego, CA, 92093-0112 USA.
			Email: jverstraete@ucsd.edu.
			Research supported by NSF-BSF award DMS-2347832.}  
\date{}

\begin{document}
\begin{abstract}
It is shown that every Eulerian digraph on $n$ vertices with more than
$(t-1)n$ arcs contains every oriented tree with $t$ edges. The digraphs
have no loops or repeated arcs, but opposite arcs are permitted. The
bound is sharp for each fixed oriented tree, as witnessed by disjoint
unions of complete bidirected graphs. Previously, such tight bounds were not known, even just for directed paths. This can be considered as a directed analog of the recently proved Erd\H os-S\'os conjecture. The result was proved by GPT-6 Astra.
\iffalse For oriented hosts, the resulting
condition on the underlying average degree is greater than $2(t-1)$,
and regular tournaments show that this threshold is sharp for a theorem
covering all oriented trees. If the tree contains a directed path with
two edges, equality in this condition suffices. In particular, a
nonempty Eulerian oriented graph of average outdegree $a$ contains a
directed path of length greater than $a$. The proof combines an
inequality for increasing families along random permutation prefixes
with an injection that glues embeddings along a prescribed edge.
A quantitative version also gives a lower bound on the number of
embeddings. \fi
\end{abstract}
\maketitle
 \section{AI Declaration Statement}
The proofs of all the results in this paper were found by GPT-6 Astra. The authors merely observed that the method of GPT-6 Astra's recent proof of the Erd\H os-S\'os conjecture  (see Adamczewski and Bloom~\cite[Appendix B.4]{Epoch})  should apply in the directed setting as well and prompted GPT-6 Astra to prove the corresponding theorems for directed graphs using the same technique. The authors have checked the elegant proof for paths (Section \ref{sec:paths}) which was the motivation for this work. They also modified the presentation, and added further explanations where appropriate. They take full responsibility for all the content in the paper.

\section{Introduction}

The relation between the density of a graph and the trees it must
contain is a central question in extremal graph theory. The theorem of
Erd\H{o}s and Gallai~\cite{EG} asserts that an undirected graph on $n$
vertices with more than $(t-1)n/2$ edges contains a path with $t$ edges.
The Erd\H{o}s--S\'os conjecture replaces the path by an arbitrary tree
with $t$ edges; see the survey of Stein~\cite{Stein} for its background
and related degree conditions. A recent proof of
the conjecture was provided by GPT-6 Astra and is described by Adamczewski and
Bloom~\cite[Appendix~B.4]{AB}. In this paper  that proof method is applied to the setting of directed graphs. Density alone cannot force even a directed path with two edges in a
general digraph: orienting every edge of a complete bipartite graph
from one part to the other gives arbitrarily large average outdegree,
while every directed path has at most one edge. Requiring the indegree
and outdegree to agree at every vertex removes this obstruction.
The resulting class of Eulerian digraphs has been studied extensively
in connection with long directed paths and cycles
\cite{BS,HMSSY,JST,KLM,KLMN}.

\bigskip

Throughout, a \emph{digraph} has no loops and at most one arc with any
given ordered pair of endpoints.  We allow digons, so both $uv$ and $vu$ may be arcs. An \emph{oriented graph} is a digraph with no pair of opposite arcs. We call a digraph \emph{Eulerian} if
\[
 d_G^+(v)=d_G^-(v)\qquad\text{for every }v\in V(G);
\]
connectivity is not required. An \emph{oriented tree} is any orientation of an undirected tree. All copies are subgraphs, and embeddings are (not necessarily induced) injective maps preserving directions of arcs. The length of a path is its number of edges.
For a digraph with $n$ vertices and $m$ arcs, let
\[
 a(G)=\frac mn
\]
for its average outdegree. If $G$ is oriented, the average degree of
its underlying undirected graph is $\overline d(G)=2a(G)$. We use these two quantities separately: the factor of two matters in the statements
below. The Erd\H{o}s-S\'{o}s Conjecture~\cite{Erdos} states that a graph of average degree more than $k$ contains every tree with $k$ edges. It is natural to propose the following analog of the Erd\H{o}s-S\'{o}s Conjecture for directed graphs. It appears not to have been explicitly stated in the literature before, perhaps because much weaker statements were not known to be true.

\begin{conjecture}\label{conj:directed-es}
    Let $t \geq 1$ and let $G$ be an eulerian digraph with $a(G) > t - 1$. Then $G$ contains every orientation of every tree with $t$ edges. 
\end{conjecture}

This is best possible, in the sense that if $G$ is obtained by doubling every edge of a clique with $t$ vertices and then taking an eulerian orientation, $G$ contains no oriented trees with $t$ edges and has $a(G) = t - 1$. Disjoint unions of these graphs give examples with $m=(t-1)n$ for
arbitrarily large multiples $n$ of $t$. Conjecture~\ref{conj:directed-es}, even for paths is nontrivial. Indeed, 
Bollob\'as and Scott~\cite{BS} conjectured that an Eulerian digraph $G$ contains a directed path of length $\Omega(a)$, where $a=a(G)$. Huang, Ma, Shapira, Sudakov and
Yuster~\cite{HMSSY} obtained a square-root bound.  Janzer, Sudakov and Tomon~\cite{JST} improved this to $\Omega(a^{1/2+1/40})$. Knierim, Larcher and Martinsson~\cite{KLM}
subsequently proved the bound $a/(\log a+1)$. Their argument builds on the cycle-decomposition
methods of Knierim, Larcher, Martinsson and Noever~\cite{KLMN}.

\bigskip

Using the ideas in the proof of the Erd\H{o}s-S\'{o}s Conjecture~\cite{Epoch}, one can prove  Conjecture~\ref{conj:directed-es}, and consequently obtain tight bounds that apply for any orientation of a tree (not just a directed path).
\begin{theorem}\label{thm:main}
Let $t\ge1$, and let $G$ be an Eulerian digraph on $n\ge1$ vertices
with $m$ arcs. If
\begin{equation}\label{eq:threshold}
 m>(t-1)n,
\end{equation}
then $G$ contains every oriented tree with $t$ edges.
\end{theorem}

The general tree problem also belongs to a broader theory of oriented
tree embeddings. Addario-Berry, Havet, Linhares Sales, Reed and
Thomass\'e~\cite{ABHLRT} studied density conditions for antidirected trees
in digraphs. Stein and Z\'arate-Guer\'en~\cite{SZ} obtained results for balanced
antidirected trees of bounded maximum degree under semidegree and
density hypotheses in large oriented graphs. Another important setting
is tournaments: K\"uhn, Mycroft and Osthus~\cite{KMO} proved Sumner's Conjecture~\cite{Sumner} for sufficiently large tournaments: for $t$ sufficiently large, every tournament on $2t$ vertices contains
every oriented tree with $t$ edges. Our hypothesis instead imposes
balance at each vertex and permits arbitrary order, density and
maximum degree.

\section{Short proof for paths}\label{sec:paths}

In this section, we give the remarkably short proof of Theorem \ref{thm:main} for paths.

\begin{proof}
Let $D$ be an Eulerian orientation of a finite simple graph with
$n$ vertices and $m$ edges. For each vertex $v$, let $\ell(v)$ denote
the maximum length of a directed path starting at $v$. Then it is sufficient to prove 
\[
    \sum_{v\in V(D)} \ell(v) \geq m.
\]
For $v\in V(D)$ and $S\subseteq V(D)\setminus\{v\}$, define
\[
    f_v(S)
    =
    \max\bigl\{
        |E(P)| :
        P \text{ is a directed path starting at }v
        \text{ in }D[S\cup\{v\}]
    \bigr\}.
\]
The function $f_v$ is increasing, and
\[
    f_v(\varnothing)=0,
    \qquad
    f_v\bigl(V(D)\setminus\{v\}\bigr)=\ell(v).
\]

Choose a uniformly random permutation of $V(D)$, and let $P_v$
be the set of vertices preceding $v$. For distinct vertices $v,w$,
put
\[
    c_v=\mathbb{E}f_v(P_v),
    \qquad
    a_{vw}=\mathbb{E}f_v(P_v\cup\{w\}),
    \qquad
    b_{vw}=\mathbb{E}f_v(P_v\setminus\{w\}).
\]
By monotonicity,
\begin{equation}
    b_{vw}\leq c_v\leq a_{vw}.
    \label{eq:monotonicity}
\end{equation}

\medskip
\noindent
\textbf{The permutation identity.}
We claim that, for every vertex $v$,
\begin{equation}
    \sum_{w\neq v}(a_{vw}-b_{vw})=\ell(v).
    \label{eq:identity}
\end{equation}
To see this, take a uniformly random permutation $\sigma$ of
$V(D)\setminus\{v\}$, and let $Q_w$ be the set of vertices preceding
$w$ in $\sigma$. For each fixed $w\neq v$, the sets $P_v\setminus\{w\}$ and $Q_w$ have the same distribution. Indeed, deleting $w$ from a uniformly
random permutation of $V(D)$ and then relabelling $v$ as $w$
produces a uniformly random permutation of $V(D)\setminus\{v\}$.
Therefore
\[
    a_{vw}-b_{vw}
    =
    \mathbb{E}_{\sigma}
    \bigl[
        f_v(Q_w\cup\{w\})-f_v(Q_w)
    \bigr].
\]
Summing over $w\neq v$ gives
\[
\begin{aligned}
    \sum_{w\neq v}(a_{vw}-b_{vw})
    &=
    \mathbb{E}_{\sigma}
    \sum_{w\neq v}
    \bigl[
        f_v(Q_w\cup\{w\})-f_v(Q_w)
    \bigr] \\
    &=
    f_v\bigl(V(D)\setminus\{v\}\bigr)-f_v(\varnothing) \\
    &=\ell(v),
\end{aligned}
\]
since the inner sum telescopes along the permutation $\sigma$.
This proves \eqref{eq:identity}.

\medskip
\noindent
\textbf{Charging the arcs.}
For every arc $v\to w$ and every
$S\subseteq V(D)\setminus\{v,w\}$, we have
\[
    f_v(S\cup\{w\})\geq 1+f_w(S).
\]
Indeed, a directed path starting at $w$ in $D[S\cup\{w\}]$
avoids $v$, so prepending the arc $v\to w$ gives a directed path
starting at $v$. The random sets $P_v\setminus\{w\}$ and $P_w\setminus\{v\}$
have the same distribution. Averaging the preceding inequality
therefore yields
\begin{equation}
    a_{vw}\geq 1+b_{wv}
    \qquad\text{for every arc }v\to w.
    \label{eq:arc}
\end{equation}

Summing \eqref{eq:arc} over all arcs, and using
$d^+(v)=d^-(v)$ to cancel the $c_v$-terms, gives
\[
\begin{aligned}
    m
    &\leq
    \sum_{v\to w}(a_{vw}-b_{wv}) \\
    &=
    \sum_{v\in V(D)}
    \Bigl[
        \sum_{w\in N^+(v)}(a_{vw}-c_v)
        +
        \sum_{w\in N^-(v)}(c_v-b_{vw})
    \Bigr] \\
    &\leq
    \sum_{v\in V(D)}\sum_{w\neq v}(a_{vw}-b_{vw}) \\
    &=
    \sum_{v\in V(D)}\ell(v).
\end{aligned}
\]
For the second inequality, use \eqref{eq:monotonicity} and the
fact that $N^+(v)$ and $N^-(v)$ are disjoint. The final equality
is \eqref{eq:identity}. Thus some vertex $v$ satisfies
\[
    \ell(v)\geq \frac{m}{n}=\frac{d}{2}.
\]
Since $\ell(v)$ is an integer, $D$ contains a directed path of
length at least $\lceil d/2\rceil$.
\end{proof}

\bigskip

According to GPT-6 Astra, if the eulerian digraph $D$ contains no digons (directed cycles of length two), then the proof can be adapted to finding a directed path of length at least $d \cdot \log 2$. It is an open question to determine whether a lower bound of $d$ is possible on the length of a longest directed path.

\section{The proof framework for trees}

We count pairs consisting of a vertex permutation and a marked
position, with the endpoints of a fixed tree edge assigned to the
first and marked entries. The prefix ending at the mark must support
an embedding of the entire tree. The count for a one-edge tree is
$m(n-1)!$. We show that adding one tree edge costs at most $n!$,
while two trees can be joined along their distinguished edge without
an additional loss. Induction gives the lower bound
\[
 \bigl(m-(t-1)n\bigr)(n-1)!.
\]
The permutation framework follows the approach described
in~\cite{AB}. The directed extension requires a comparison between
prefixes marked at in-neighbours and prefixes marked at out-neighbours.
We prove this comparison using the equal sizes of the two
neighbourhoods and the Margulis--Russo differentiation
identity~\cite{Margulis,Russo}. All of the needed inequalities, including
the permutation injection for gluing, are proved below.

\section{Two inequalities for permutation prefixes}\label{sec:prefixes}

For a permutation $\sigma$ of a finite set $U$ and $u\in U$, let
$S_{<u}(\sigma)$ denote the set of entries preceding $u$, and put
\[
 S_{\le u}(\sigma)=S_{<u}(\sigma)\cup\{u\}.
\]
A family $\cF\subseteq2^U$ is \emph{increasing} if $A\in\cF$ and
$A\subseteq B\subseteq U$ imply $B\in\cF$.

\begin{lemma}[Balanced prefixes]\label{lem:balance}
Let $I,J\subseteq U$ satisfy $|I|=|J|$, and let
$\cF\subseteq2^U$ be increasing. If $\sigma$ is a uniform permutation
of $U$, then
\begin{equation}\label{eq:balance}
 \E\left[
 \sum_{u\in I}\ind_{\{S_{\le u}(\sigma)\in\cF\}}
 -\sum_{w\in J}\ind_{\{S_{<w}(\sigma)\in\cF\}}
 \right]\le1.
\end{equation}
If $I\cap J=\varnothing$, the inequality is strict.
\end{lemma}

\begin{proof}
For $p\in[0,1]$, let $R_p$ be a random subset of $U$ in which each
element is present independently with probability $p$. Define
\begin{align*}
 f(p)&=\Prob(R_p\in\cF),\\
 f_u^+(p)&=\Prob(R_p\cup\{u\}\in\cF),&
 f_u^-(p)&=\Prob(R_p\setminus\{u\}\in\cF),
\end{align*}
and set $D_u(p)=f_u^+(p)-f_u^-(p)$. Since $\cF$ is increasing,
$D_u(p)\ge0$. Conditioning on the membership of $u$ gives
\begin{equation}\label{eq:conditioning}
 f_u^+(p)=f(p)+(1-p)D_u(p),\qquad
 f_u^-(p)=f(p)-pD_u(p).
\end{equation}
We also have the differentiation identity
\begin{equation}\label{eq:russo}
 f'(p)=\sum_{u\in U}D_u(p).
\end{equation}
For completeness, write $\Prob(R\in\cF)$ as a multilinear polynomial
in the individual membership probabilities $(p_u)_{u\in U}$.
The partial derivative in $p_u$, evaluated when all coordinates equal
$p$, is $D_u(p)$. Differentiating along the diagonal gives
\eqref{eq:russo}. This is the finite increasing-event form of the
Margulis--Russo identity~\cite{Margulis,Russo}.

By $|I|=|J|$ and~\eqref{eq:conditioning},
\begin{align}
 \sum_{u\in I}f_u^+(p)-\sum_{w\in J}f_w^-(p)
 &=(1-p)\sum_{u\in I}D_u(p)+p\sum_{w\in J}D_w(p)\notag\\
 &\le\sum_{u\in U}D_u(p)=f'(p).
 \label{eq:integrand}
\end{align}
Indeed, the coefficient of any $D_u(p)$ is one of $0,p,1-p,1$.
Thus overlapping sets $I$ and $J$ cause no difficulty. Assign independent uniform labels in $[0,1]$ to the elements of $U$
and order the elements by increasing label. This gives a uniform
permutation. Conditional on the label of $u$ being $p$, the other
elements preceding $u$ form an independent $p$-random subset.
The expectation in~\eqref{eq:balance} therefore equals
\[
 \int_0^1\left(
 \sum_{u\in I}f_u^+(p)-\sum_{w\in J}f_w^-(p)
 \right)\,dp
 \le f(1)-f(0)\le1.
\]

Suppose now that $I\cap J=\varnothing$. If $\cF$ is empty or is
$2^U$, the expectation is zero. Otherwise $f(1)-f(0)=1$, and
\eqref{eq:integrand} improves to
\[
 \sum_{u\in I}f_u^+(p)-\sum_{w\in J}f_w^-(p)
 \le\max\{p,1-p\}\,f'(p).
\]
Since $f'\ge0$, $\int_0^1f'(p)\,dp=1$ and
$\max\{p,1-p\}<1$ on $(0,1)$, its integral is strictly less than one.
This proves the strict assertion as well. When $U$ is empty, both
sums vanish.
\end{proof}

The next lemma does not require the families to be increasing. For a
word $w$ with distinct letters, let $\supp(w)$ be its set of letters,
and let $w_{\le s}$ be its first $s$ letters. If $|U|=N$ and
$\cA\subseteq2^U$, define
\[
 c_U(\cA)=\bigl|\{(w,s): w\text{ is a permutation of }U,\
 0\le s\le N,\ \supp(w_{\le s})\in\cA\}\bigr|.
\]
There are $(N+1)N!=(N+1)!$ pairs $(w,s)$, including pairs marked at
the empty prefix.

\begin{lemma}[Prefix gluing]\label{lem:prefix-glue}
Let $\cA,\cB,\cC\subseteq2^U$ satisfy $\cC\subseteq\cA$. Suppose that
\begin{equation}\label{eq:disjoint-glue}
 A\in\cA,\quad B\in\cB,\quad A\cap B=\varnothing
 \quad\Longrightarrow\quad A\cup B\in\cC.
\end{equation}
Then
\begin{equation}\label{eq:prefix-glue}
 c_U(\cA)+c_U(\cB)\le(N+1)!+c_U(\cC).
\end{equation}
\end{lemma}

\begin{proof}
We inject the pairs whose marked prefix is in $\cA\setminus\cC$
into the pairs whose marked prefix is outside $\cB$.
Take such an input $(w,s)$. Let $R$ be the shortest prefix word of
$w$ whose support belongs to $\cA\setminus\cC$. Since the marked
prefix qualifies, $|R|\le s$. Write
\[
 w=RXY,\qquad s=|R|+|X|,
\]
and define
\begin{equation}\label{eq:injection}
 (RXY,|R|+|X|)\longmapsto(XRY,|X|).
\end{equation}
If $\supp(X)\in\cB$, then~\eqref{eq:disjoint-glue}, applied to
$\supp(R)$ and $\supp(X)$, gives $\supp(RX)\in\cC$. This contradicts
the condition on the marked prefix of the input. Thus the output is
outside $\cB$. To recover the input, the output mark first determines the word $X$.
In the remaining suffix $RY$, take the shortest prefix whose support
lies in $\cA\setminus\cC$. This is exactly $R$: it qualifies, and
each of its shorter prefixes was a shorter prefix of the original
word $w$. This determines $R$, then $Y$, and finally the input mark.
The map is therefore injective. Empty words are allowed throughout.
Since $\cC\subseteq\cA$, the injection gives
\[
 c_U(\cA)-c_U(\cC)\le(N+1)!-c_U(\cB),
\]
which is~\eqref{eq:prefix-glue}.
\end{proof}

\section{Rooted states and leaf extensions}\label{sec:states}

Fix an Eulerian digraph $G$ with $n\ge2$ vertices and $m$ arcs, and put
\begin{equation}\label{eq:MW}
 M=m(n-1)!,\qquad W=n!.
\end{equation}
All counts in the remainder of the proof refer to this fixed host.

\begin{definition}\label{def:state}
Let $S$ be an oriented tree with a distinguished edge $e$, and choose
an ordering $(x,y)$ of the endpoints of $e$. A \emph{state} is a pair
$(\pi,j)$, where $\pi=(v_1,\ldots,v_n)$ is a permutation of $V(G)$
and $2\le j\le n$. It is \emph{good} for $(S,e;x,y)$ if there is an
embedding $\varphi:S\hookrightarrow G$ such that
\begin{equation}\label{eq:good}
 \varphi(x)=v_1,\qquad \varphi(y)=v_j,\qquad
 \varphi(V(S))\subseteq\{v_1,\ldots,v_j\}.
\end{equation}
Each state is counted once, regardless of the number of witnessing
embeddings.
\end{definition}

The ordering $(x,y)$ need not agree with the direction of $e$.
Moreover, the number of good states is independent of this ordering.
Swapping the entries of $\pi$ in positions $1$ and $j$ gives a
bijection from good states for $(S,e;x,y)$ to good states for
$(S,e;y,x)$. The same embedding witnesses the new state and has the
same image set. No arc is reversed. We may therefore denote this
count by $C_G(S,e)$, or simply by $C(S,e)$. For the one-edge tree,
\begin{equation}\label{eq:base}
 C(e,e)=M.
\end{equation}
Indeed, choose one of the $m$ host arcs as the image of the oriented
edge, fix the image of $x$ in the first position, and permute the other
$n-1$ vertices. The position of the image of $y$ determines the mark.

\begin{lemma}[Leaf extension]\label{lem:leaf}
Let $f=uy$ be an edge of an oriented tree $S$. Add a new vertex $z$
and an edge $e=uz$ to obtain an oriented tree $T$. Both $f$ and $e$
may have either direction. Then
\begin{equation}\label{eq:leaf}
 C(S,f)\le C(T,e)+W.
\end{equation}
If $G$ is oriented and $f,e$ form a directed path with two edges,
then the inequality is strict.
\end{lemma}

\begin{proof}
Order the root endpoints as $(u,y)$ for $S$ and $(u,z)$ for $T$.
Fix the image $v$ of $u$, and set $U_v=V(G)\setminus\{v\}$.
Let $\cF_v\subseteq2^{U_v}$ consist of the sets $A$ for which
$G[A\cup\{v\}]$ contains an embedding of $S$ sending $u$ to $v$.
This is an increasing family. Let $I_v$ be the permitted images of $y$ under the direction of $f$,
and let $J_v$ be the permitted images of $z$ under the direction of $e$:
\[
 I_v=\begin{cases}N_G^+(v),&u\to y,\\N_G^-(v),&y\to u,\end{cases}
 \qquad
 J_v=\begin{cases}N_G^+(v),&u\to z,\\N_G^-(v),&z\to u.\end{cases}
\]
The Eulerian condition gives $|I_v|=|J_v|$ in every case. The two
sets can coincide or overlap, as permitted in
Lemma~\ref{lem:balance}. For a uniform permutation $\sigma$ of $U_v$, define
\[
 X_v=\sum_{a\in I_v}\ind_{\{S_{\le a}(\sigma)\in\cF_v\}},
 \qquad
 Y_v=\sum_{b\in J_v}\ind_{\{S_{<b}(\sigma)\in\cF_v\}}.
\]
Every good state for $(S,f)$ with first entry $v$ contributes to
$X_v$. This implication need not be an equivalence, because an
embedding supported on the prefix counted by $X_v$ might send $y$
to a vertex other than its marked entry. Conversely, every contribution to $Y_v$ supplies a good state for
$(T,e)$. Choose an embedding of $S$ in
$G[S_{<b}(\sigma)\cup\{v\}]$ sending $u$ to $v$, and extend it by
sending $z$ to $b$. The new image is injective, and the new arc has
the required direction because $b\in J_v$. Thus
\[
 C(S,f)\le(n-1)!\sum_{v\in V(G)}\E X_v,\qquad
 C(T,e)\ge(n-1)!\sum_{v\in V(G)}\E Y_v.
\]
Subtracting and applying Lemma~\ref{lem:balance} gives
\[
 C(S,f)-C(T,e)
 \le(n-1)!\sum_{v\in V(G)}\E(X_v-Y_v)
 \le n!=W.
\]
If $G$ is oriented and $f,e$ form a directed path, then $I_v$ and
$J_v$ are the disjoint sets $N_G^+(v)$ and $N_G^-(v)$, in some order.
The strict assertion of Lemma~\ref{lem:balance} applies at every
vertex, giving strict inequality in~\eqref{eq:leaf}.
\end{proof}

\section{Gluing and the main counting inequality}\label{sec:counting}

\begin{lemma}[Gluing along an edge]\label{lem:root-glue}
Suppose $T=T_1\cup T_2$ as oriented trees, where
\[
 V(T_1)\cap V(T_2)=\{x,y\},\qquad
 E(T_1)\cap E(T_2)=\{e\},\qquad e=x\to y.
\]
Then
\begin{equation}\label{eq:root-glue}
 C(T_1,e)+C(T_2,e)\le M+C(T,e).
\end{equation}
\end{lemma}

\begin{proof}
Fix a host arc $a\to b$ and prescribe $x\mapsto a$, $y\mapsto b$.
Set $U=V(G)\setminus\{a,b\}$ and $N=n-2$. States with this root
assignment correspond exactly to the pairs
\[
 (w,s)\longleftrightarrow
 \bigl((a,w_{\le s},b,w_{>s}),s+2\bigr),\qquad 0\le s\le N,
\]
where $w$ is a permutation of $U$. Let $\cA,\cB,\cC\subseteq2^U$ consist of the sets supporting embeddings
of $T_1,T_2,T$, respectively, after adjoining $a,b$, with the prescribed
images of $x,y$. Restriction gives $\cC\subseteq\cA$. If
$A\in\cA$, $B\in\cB$ and $A\cap B=\varnothing$, the two embeddings
agree on $x,y$ and have disjoint images elsewhere. Their union is
an injective embedding of $T$, so $A\cup B\in\cC$. Lemma~\ref{lem:prefix-glue} applies to these families. Its constant is
$(N+1)!=(n-1)!$. Summing the resulting inequality over all $m$ choices
of the host arc $a\to b$ proves~\eqref{eq:root-glue}.
\end{proof}

\begin{theorem}[Rooted counting inequality]\label{thm:rooted}
For every oriented tree $T$ with $t\ge1$ edges and every
$e\in E(T)$,
\begin{equation}\label{eq:rooted}
 C(T,e)\ge M-(t-1)W
 =\bigl(m-(t-1)n\bigr)(n-1)!.
\end{equation}
\end{theorem}

\begin{proof}
We induct on $t$, simultaneously over all oriented trees and all
choices of distinguished edge. The case $t=1$ is~\eqref{eq:base}.
Suppose $t\ge2$. If $e=uz$ is pendant, with leaf vertex $z$, delete $z$ and $e$ to
obtain $S$. There is an edge $f$ of $S$ incident with $u$, since
$t\ge2$. By Lemma~\ref{lem:leaf} and induction,
\[
 C(T,e)\ge C(S,f)-W
 \ge M-(t-2)W-W=M-(t-1)W.
\]

If $e$ is not pendant, deleting it produces two components, each
containing at least one edge. Let $T_1$ consist of the first component
together with $e$, and let $T_2$ consist of the second component
together with $e$. The orientations are inherited from $T$. Each
$T_i$ has fewer than $t$ edges; the two trees intersect exactly in
$e$ and its endpoints. Write $t_i=|E(T_i)|$, so that
$t_1+t_2=t+1$. By Lemma~\ref{lem:root-glue} and induction,
\begin{align*}
 C(T,e)
 &\ge C(T_1,e)+C(T_2,e)-M\\
 &\ge \bigl(M-(t_1-1)W\bigr)+\bigl(M-(t_2-1)W\bigr)-M\\
 &=M-(t-1)W.
\end{align*}
This completes the induction.
\end{proof}

\begin{proof}[Proof of Theorem~\ref{thm:main}]
The hypothesis implies $m>0$, and hence $n\ge2$. Fix an edge $e$ of
the prescribed oriented tree $T$. Theorem~\ref{thm:rooted} gives
\[
 C(T,e)\ge\bigl(m-(t-1)n\bigr)(n-1)!>0.
\]
A good state supplies an embedding of $T$. 
\end{proof}

\subsection{Counting embeddings}

Let $\emb(T,G)$ denote the number of injective maps from $V(T)$ to
$V(G)$ that preserve all arcs. Thus automorphisms of $T$ are not
identified in this count.

\begin{proposition}\label{prop:embeddings}
For every oriented tree $T$ with $t\ge1$ edges and every Eulerian
digraph $G$ on $n\ge2$ vertices with $m$ arcs,
\begin{equation}\label{eq:embeddings}
 \emb(T,G)\ge t\bigl(m-(t-1)n\bigr).
\end{equation}
\end{proposition}

\begin{proof}
Fix an ordering $(x,y)$ of the endpoints of an edge $e$ of $T$.
For any particular embedding $\varphi$, a state it witnesses must
have $\varphi(x)$ first and $\varphi(y)$ after the other $t-1$
vertices of $\varphi(V(T)\setminus\{x,y\})$. Among permutations of
the remaining $n-1$ host vertices, each of the $t$ images of
$V(T)\setminus\{x\}$ is equally likely to occur last. Thus $\varphi$
witnesses exactly $(n-1)!/t$ states. Counting pairs of a good state
and a witnessing embedding gives
\[
 C(T,e)\le\emb(T,G)\frac{(n-1)!}{t}.
\]
Combining this with Theorem~\ref{thm:rooted} proves the result.
\end{proof}

\section{Further consequences and questions}\label{sec:questions}

\subsection{The undirected specialization}

Theorem~\ref{thm:main} contains the Erd\H{o}s--S\'os statement as a
special case. Given an undirected graph $H$ with $n$ vertices and $q$
edges, replace each edge by its two opposite arcs to form a digraph
$D$. This digraph is Eulerian and has $2q$ arcs.
If $q>(t-1)n/2$, Theorem~\ref{thm:main} embeds any chosen orientation
of a prescribed $t$-edge tree into $D$. Forgetting
the arc directions gives the desired copy in $H$. This also explains
why allowing opposite arcs is natural in the main theorem and why
the normalization in~\eqref{eq:threshold} has the stated constant.

\subsection{A possible factor of two for directed paths}

For arbitrary oriented trees the regular-tournament star example
makes our main theorem sharp. That example does not
establish sharpness for directed paths. We mentioned a natural strengthening of
Theorem~\ref{thm:main} for directed paths in Eulerian digraphs with no digons:

\begin{conjecture}\label{conj:path}
Every Eulerian oriented graph $G$ has a directed path of length at least $\overline d(G)$.
\end{conjecture}

Since path length is integral, this asks for a path with at least
$\lceil\overline d(G)\rceil$ edges. The constant one in front of
$\overline d(G)$ could not be increased: a regular tournament on
$2q+1$ vertices has underlying average degree $2q$, and no path can
have more than $2q$ edges. Jackson~\cite{Jackson} proved that an oriented graph with minimum
indegree and minimum outdegree at least $q$ contains a directed path
with $2q$ edges. Consequently, Conjecture~\ref{conj:path} holds when
the Eulerian oriented graph has a regular underlying graph. The issue
is to obtain the same conclusion from average degree when the vertex
degrees vary. Theorem \ref{thm:main} gives approximately half the conjectured
length and does not settle this strengthening. As mentioned, GPT-6 Astra gives a proof of the lower bound $\overline{d}(G)\cdot \log 2$ by optimizing the proof given in Section \ref{sec:paths}. One cannot obtain that factor of two simply by replacing the loss
$W$ in every directed leaf extension by $W/2$. The following exact
calculation already prevents such an improvement for a two-edge path.

\begin{proposition}\label{prop:loss}
Let $G$ be an oriented graph in which every vertex has indegree and
outdegree $q\ge1$. Let $P$ be the directed path $x\to u\to z$,
rooted at $e=u\to z$, and let $f=x\to u$. Then
\begin{equation}\label{eq:loss}
 C(f,f)-C(P,e)=\frac{q}{q+1}\,n!.
\end{equation}
\end{proposition}

\begin{proof}
Order the root endpoints as $(u,z)$ and fix the image $v$ of $u$.
For each $b\in N_G^+(v)$, the state marked at $b$ fails to support
$P$ precisely when no member of $N_G^-(v)$ precedes $b$ in the
permutation of $V(G)\setminus\{v\}$. Since $G$ is oriented,
$b\notin N_G^-(v)$. Among the $q+1$ vertices in
$N_G^-(v)\cup\{b\}$, each is equally likely to come first. The
failure probability is therefore $1/(q+1)$.

There are $n$ choices of $v$, $q$ choices of $b$, and $(n-1)!$
permutations after $v$. Hence
\[
 C(P,e)=nq(n-1)!\left(1-\frac1{q+1}\right).
\]
Since $C(f,f)=nq(n-1)!$, subtracting proves~\eqref{eq:loss}.
\end{proof}

\subsection{Cycles and prescribed roots}

The existence of a long directed path does not by itself produce a
cycle of comparable length. The linear cycle problem associated
with Bollob\'as and Scott~\cite{BS} therefore remains separate from
our argument. The square-root cycle bound of Huang, Ma, Shapira,
Sudakov and Yuster~\cite{HMSSY} was improved by Ai, Gutin, He and
Yeo~\cite{AGHY} to $\sqrt{2a(G)}-3/2$. The cycle-decomposition theorem
of Knierim, Larcher, Martinsson and Noever~\cite{KLMN} also gives a
bound of order $a(G)/\log\Delta$ in terms of the maximum degree.
The permutation proof above does not provide the additional edge
needed to close the embedded path. It would also be interesting to obtain a linear bound for a path
starting at an arbitrary specified vertex of a connected Eulerian
digraph. The rooted states used here fix positions in a permutation,
but the final count sums over the images of the root. Thus
Theorem~\ref{thm:rooted} does not imply such a vertex-prescribed
statement. The result of Janzer, Sudakov and Tomon~\cite{JST}
addresses that stronger form with a smaller power of the average
outdegree.

\end{document}